\documentclass[pre]{revtex4-2}

\usepackage{amsthm, mathtools}
\usepackage{cases}

\theoremstyle{plain}
\newtheorem{theorem}{Theorem}[section]

\newtheorem{proposition}{Proposition}[section]
\newtheorem{corollary}{Corollary}[section]

\theoremstyle{definition}
\newtheorem{remark}{Remark}[section]

\newcommand{\Prob}[1]{\mathrm{Pr}\left(#1\right)} 
\newcommand{\E}[1]{E\left[#1\right]} 
\newcommand{\Var}[1]{V\left[#1\right]} 

\newcommand{\GaussFsymbol}{{}_2F_1} 
\newcommand{\GaussF}[4]{\GaussFsymbol\left(#1, #2; #3; #4\right)} 

\begin{document}
\title{Complete corrected formula for generating functions of the hypergeometric distribution}
\author{Ken Yamamoto}
\affiliation{Faculty of Science, University of the Ryukyus, Nishihara, Okinawa 903-0213, Japan}

\begin{abstract}
The hypergeometric distribution is a popular distribution, whose properties have been extensively investigated.
Generating functions of this distribution, such as the probability-generating function, the moment-generating function, and the characteristic function, are known to involve the Gauss hypergeometric function.
We elucidate that the existing formula of generating functions is incomplete and provide the corrected formula.
In view of the utility and applicability of the hypergeometric distribution, the exact and correct foundation is crucially important.
\end{abstract}

\maketitle



\section{Introduction}
The hypergeometric distribution is a useful and popular discrete probability distribution in probability theory and statistics.
This distribution describes the number of successes in a sample drawn from a finite population without replacement.
The hypergeometric distribution has been applied to biology~\cite{Fury}, complex networks~\cite{Casiraghi}, and blockchain technology~\cite{Hafid}.
Furthermore, the hypergeometric distribution provides the theoretical basis for Fisher's exact test~\cite{Rivals} and the Mantel--Haenszel test~\cite{Wittes} used in medical statistics, and mark--recapture methods~\cite{Seber} in ecology.
A simple example of the hypergeometric distribution is picking colored balls from an urn~\cite{Johnson, Balakrishnan}.
Let an urn contain $K$ white balls and $N-K$ black balls.
We let $X$ denote the number of white balls in $n$ balls picked randomly.
The probability that $k$ balls are white is given by
\begin{equation}
\Prob{X=k}=\left.\binom{K}{k}\binom{N-K}{n-k}\middle/\binom{N}{n}\right..
\label{eq:pmf}
\end{equation}
This is the probability mass function of the hypergeometric distribution.
The support of this distribution, i.e., the set of $k$ with $\Prob{X=k}>0$, consists of integers within $\max\{0, n+K-N\}\le k\le\min\{n, K\}$.
If we consider the factorial of negative integers as infinity, Eq.~\eqref{eq:pmf} is correct for all integers $k$.

The hypergeometric distribution is related to several distributions.
When picking balls \emph{with replacement}, the number of white balls follows the binomial distribution~\cite{Johnson}, which is simpler than the hypergeometric distribution.
When we pick balls one by one without replacement until the number of black balls reaches a given value, the number of white balls follows the negative hypergeometric distribution~\cite{Mukhopadhyay}.
The multivariate hypergeometric distribution~\cite{Bishop} appears when the balls have more than two colors.
In the P\'olya urn model~\cite{Mahmoud}, at each step, one ball is drawn, its color is observed, and the ball is returned in the urn with additional balls of the same color.
In this model, the number of white balls follows a distribution referred to as the beta-binomial distribution~\cite{Johnson}.

Many publications, e.g., \cite{Johnson, Balakrishnan, Bishop, Patel}, contain the moment-generating function, the characteristic function, and the probability-generating function of the hypergeometric distribution.
The most fundamental function among them is the probability-generating function
\[
G_X(z)\coloneqq \E{z^X}=\sum_{k=0}^\infty \Prob{X=k}z^k.
\]
Using $G_X$, the moment-generating function is expressed as $M_X(t)=G_X(e^t)$, the characteristic function as $\varphi_X(t)=G_X(e^{it})$ ($i=\sqrt{-1}$), and the cumulant-generating function as $K_X(t)=\ln G_X(e^t)$.
The existing formula for the probability-generating function of the hypergeometric distribution is
\begin{equation}
G_X(z)=\frac{(N-n)!(N-K)!}{N!(N-K-n)!}\GaussF{-n}{-K}{N-K-n+1}{z}.
\label{eq:G}
\end{equation}
Here, $\GaussFsymbol$ is the Gauss hypergeometric function defined by
\[
\GaussF{a}{b}{c}{z}\coloneqq\sum_{k=0}^\infty \frac{(a)_k(b)_k}{(c)_k}\frac{z^k}{k!},
\]
where $(a)_k$ is the Pochhammer symbol:
\[
(a)_0=1,\quad
(a)_k=a(a+1)\cdots(a+k-1).
\]

The Gauss hypergeometric function $\GaussF{a}{b}{c}{z}$ diverges when $c$ is zero or negative integer because $(c)_k=0$ for $k\ge|c|+1$.
Therefore, Eq.~\eqref{eq:G} is not properly defined when $n\ge N-K+1$.
The factorial $(N-K-n)!$ also diverges for such $n$, and Eq.~\eqref{eq:G} takes an indeterminate form ``$\infty/\infty$.''
The condition $n\ge N-K+1$ indicates that the number $n$ of balls drawn is greater than the number of black balls.
In this study, we aim to establish the complete formula for $G_X(z)$ by providing the determinate expression for $n\ge N-K+1$.

\begin{theorem}[Main result]
The complete probability-generating function of $X$ becomes
\begin{subnumcases}{\label{eq:thm1} G_X(z)=}
\frac{(N-n)!(N-K)!}{N!(N-K-n)!}\GaussF{-n}{-K}{N-K-n+1}{z} & $n\le N-K$, \label{eq:thm1-1}\\
\frac{n!K!}{N!(n+K-N)!}z^{n+K-N}\GaussF{n-N}{K-N}{n+K-N+1}{z} & $n\ge N-K+1$. \label{eq:thm1-2}
\end{subnumcases}
\label{thm1}
\end{theorem}

Equation~\eqref{eq:thm1-2} is the correct determinate formula for $n\ge N-K+1$.
Although this result might be trivial for expert mathematicians, the explicit expression will be useful for students, engineers, practitioners, and programmers who are not necessarily familiar with mathematics.

\begin{remark}
As the first two arguments of the hypergeometric function, $-n$ and $-K$ in Eq.~\eqref{eq:thm1-1}, and $n-N$ and $K-N$ in Eq.~\eqref{eq:thm1-2}, are negative integers, the infinite sum of $\GaussFsymbol$ becomes finite.
Therefore, $\GaussFsymbol$ in Eq.~\eqref{eq:thm1} becomes a polynomial of $z$ for any $N$, $K$, and $n$, whose degree is $\min\{n, K\}$ for Eq.~\eqref{eq:thm1-1} and $\min\{N-n, N-K\}$ for Eq.~\eqref{eq:thm1-2}.
In both Eqs.~\eqref{eq:thm1-1} and \eqref{eq:thm1-2}, $G_X(z)$ has degree $\min\{n, K\}$; the degree of Eq.~\eqref{eq:thm1-2} is calculated as $n+K-N+\min\{N-n, N-K\}=\min\{n, K\}$.
\end{remark}

\begin{remark}\label{remark:1.2}
We set the range of Eq.~\eqref{eq:thm1-2} as $n\ge N-K+1$ in order to supplement the range $n\le N-K$ of Eq.~\eqref{eq:thm1-1}.
However, Eq.~\eqref{eq:thm1-2} is defined even for $n=N-K$, and both Eqs.~\eqref{eq:thm1-1} and \eqref{eq:thm1-2} for $n=N-K$ have the same form
\[
G_X(z)=\frac{n!K!}{N!}\GaussF{-n}{-K}{1}{z}.
\]
See Remark~\ref{remark:overlap} and Fig.~\ref{fig1} for this overlap in detail.
\end{remark}

The author believes that it is worth presenting the moment-generating function and characteristic function of the hypergeometric random variable $X$ explicitly, because they appear more frequently than the probability-generating function in many publications.

\begin{corollary}[Moment-generating function and characteristic function]
The moment-generating function $M_X(t)=G_X(e^t)$ and characteristic function $\varphi_X(t)=G_X(e^{it})$ of $X$ are
\[
M_X(t)=
\begin{dcases}
\frac{(N-n)!(N-K)!}{N!(N-K-n)!}\GaussF{-n}{-K}{N-K-n+1}{e^t} & n\le N-K,\\
\frac{n!K!}{N!(n+K-N)!}z^{n+K-N}\GaussF{n-N}{K-N}{n+K-N+1}{e^t} & n\ge N-K+1,
\end{dcases}
\]
and
\[
\varphi_X(t)=
\begin{dcases}
\frac{(N-n)!(N-K)!}{N!(N-K-n)!}\GaussF{-n}{-K}{N-K-n+1}{e^{it}} & n\le N-K,\\
\frac{n!K!}{N!(n+K-N)!}z^{n+K-N}\GaussF{n-N}{K-N}{n+K-N+1}{e^{it}} & n\ge N-K+1,
\end{dcases}
\]
respectively.
\end{corollary}

\section{Proofs of Theorem~\ref{thm1}}
We do not prove Eq.~\eqref{eq:thm1-1}, which is a well-known existing result and present only the proof of Eq.~\eqref{eq:thm1-2} for the $n \ge N-K+1$ case.
We provide two proofs.

\begin{proof}[First proof]
When $n\ge N-K+1$, the lower bound of the support of $X$ becomes $\max\{0, n+K-N\}=n+K-N$.
The probability-generating function can be computed as follows.
\begin{align*}
G_X(z)&=\frac{n!(N-n)!}{N!}\sum_{k=n+K-N}^{\infty}\frac{K!}{k!(K-k)!}\frac{(N-K)!}{(n-k)!(N-K-n+k)!}z^k\\
&=\frac{n!K!}{N!(n+K-N)!}z^{n+K-N}\sum_{j=0}^\infty \frac{(n+K-N)!}{(n+K-N+j)!}\frac{(N-n)!}{(N-n-j)!}\frac{(N-K)!}{(N-K-j)!}\frac{z^j}{j!}\\
&=\frac{n!K!}{N!(n+K-N)!}z^{n+K-N}\sum_{j=0}^\infty\frac{(n-N)_j(K-N)_j}{(n+K-N+1)_j}\frac{z^j}{j!}\\
&=\frac{n!K!}{N!(n+K-N)!}z^{n+K-N}\GaussF{n-N}{K-N}{n+K-N+1}{z},
\end{align*}
where $j=k-n-K+N$ is introduced in the second equality, and relations
\begin{align*}
\frac{(M+j)!}{M!}&=(M+1)(M+2)\cdots(M+j)=(M+1)_j,\\
\frac{M!}{(M-j)!}&=M(M-1)\cdots(M-j+1)=(-1)^j(-M)(-M+1)\cdots(-M+j-1)=(-1)^j(-M)_j
\end{align*}
are used in the third equality.

\end{proof}

\begin{proof}[Second proof]
We introduce the random variable $X'$ as the number of \emph{black} balls left in the urn.
The number of white balls left becomes $N-n-X'$, and the relation $K=X+(N-n-X')$ holds for the total number of white balls.
Therefore, we obtain
\[
X=X'-N+K+n.
\]

Picking $n$ balls from the urn is equivalent to choosing $N-n\eqqcolon n'$ balls to be left in the urn.
That is, $X'$ follows a hypergeometric distribution representing the number of black balls in randomly selected $n'$ balls, out of total $K'=N-K$ black balls.
When $n\ge N-K+1$, $n'$ satisfies $n'\le N-K'-1<N-K'$.
Hence, the probability-generating function of $X'$ is expressed in the form of Eq.~\eqref{eq:thm1-1} with $(n, K)$ replaced by $(n', K')$:
\begin{align*}
G_{X'}(z)&=\frac{(N-n')!(N-K')!}{N!(N-K'-n')!}\GaussF{-n'}{-K'}{N-K'-n'+1}{z}\\
&=\frac{n!K!}{N!(n+K-N)!}\GaussF{n-N}{K-N}{n+K-N+1}{z}.
\end{align*}
In conclusion,
\begin{align*}
G_X(z)&=\E{z^X}=\E{z^{X'+n+K-N}}=z^{n+K-N}G_{X'}(z)\\
&=\frac{n!K!}{N!(n+K-N)!}z^{n+K-N}\GaussF{n-N}{K-N}{n+K-N+1}{z}.
\end{align*}
\end{proof}

\begin{remark}
The Gauss hypergeometric function $\GaussF{a}{b}{c}{z}$ cannot be defined for $c=0,-1,-2,\ldots$, but the limit
\begin{equation}
\lim_{c\to-m}\frac{\GaussF{a}{b}{c}{z}}{\Gamma(c)}=\frac{(a)_{m+1}(b)_{m+1}}{(m+1)!}z^{m+1}\GaussF{a+m+1}{b+m+1}{m+2}{z}
\label{eq:scaled}
\end{equation}
holds for nonnegative integer $m$ (see~\cite{Olver}), where $\Gamma$ is the gamma function.
In this sense, $\GaussF{a}{b}{c}{z}/\Gamma(c)$ exists for any $c$.
Using this property, we can attain Eq.~\eqref{eq:thm1-2} from Eq.~\eqref{eq:thm1-1} as a formality by taking a suitable limit.
Equation~\eqref{eq:thm1-1} can be rewritten as
\[
G_X(z)=\frac{(N-n)!(N-K)!}{N!}\lim_{\varepsilon\to0}\frac{\GaussF{-n}{-K}{N-K-n+1+\varepsilon}{z}}{\Gamma(N-K-n+1+\varepsilon)}.
\]
The right-hand side remains finite even for $n\ge N-K+1$, which becomes
\[
\frac{(N-n)!(N-K)!}{N!}\frac{(-n)_{n+K-N}(-K)_{n+K-N}}{(n+K-N)!}z^{n+K-N}\GaussF{K-N}{n-N}{n+K-N+1}{z},
\]
owing to Eq.~\eqref{eq:scaled}.
By using
\[
(-n)_{n+K-N}=(-n)(-n+1)\cdots(K-N-1)=(-1)^{n+K-N}\frac{n!}{(N-K)!}
\]
and
\[
(-K)_{n+K-N}=(-K)(-K+1)\cdots(n-N-1)=(-1)^{n+K-N}\frac{K!}{(N-n)!},
\]
Eq.~\eqref{eq:thm1-2} is obtained.
\end{remark}

\section{Alternative expression}
As stated in the second proof in the previous section, picking $n$ balls is equivalent to selecting $N-n$ balls to be left in the urn.
It is expected that a kind of symmetry for $n$ and $N-n$ exists in $G_X(z)$.
However, Eq.~\eqref{eq:thm1} does not directly display this symmetry.
In fact, Eq.~\eqref{eq:thm1} switches its form at $n=N-K$, not $n=N/2$.
We seek an alternative expression of $G_X(z)$ in this section, focusing on this symmetry.

\begin{corollary}
The probability-generating function $G_X(z)$ in Eq.~\eqref{eq:thm1} is rewritten in the following form.
\begin{enumerate}
\renewcommand{\labelenumi}{$\mathrm{(\alph{enumi})}$}
\renewcommand{\theenumi}{\labelenumi}
\item\label{itm:cor1} When $K\le N/2$,
\begin{subnumcases}{\label{eq:cor-1}G_X(z)=}
\frac{(N-n)!(N-K)!}{N!(N-K-n)!}\GaussF{-n}{-K}{N-K-n+1}{z} & $n\le \dfrac{N}{2}$,\label{eq:cor-1-1}\\
\frac{n!(N-K)!}{N!(n-K)!}z^K\GaussF{n-N}{-K}{n-K+1}{z^{-1}} & $n\ge \dfrac{N}{2}$.\label{eq:cor-1-2}
\end{subnumcases}
\item\label{itm:cor2} When $K\ge N/2$,
\begin{subnumcases}{\label{eq:cor-2}G_X(z)=}
\frac{(N-n)!K!}{N!(K-n)!}z^n\GaussF{-n}{K-N}{K-n+1}{z^{-1}} & $n\le \dfrac{N}{2}$,\label{eq:cor-2-1}\\
\frac{n!K!}{N!(n+K-N)!}z^{n+K-N}\GaussF{n-N}{K-N}{n+K-N+1}{z} & $n\ge \dfrac{N}{2}$.\label{eq:cor-2-2}
\end{subnumcases}
\end{enumerate}
\label{cor}
\end{corollary}

As a formality, Eq.~\eqref{eq:cor-1-2} is formed from Eq.~\eqref{eq:cor-1-1} by replacing $n$ with $N-n$ and $z$ with $z^{-1}$, and multiplying $z^K$.
The same replacement rule is applicable to Eq.~\eqref{eq:cor-2}.
The factor $z^K$ represents that the symmetry is incomplete.

\begin{proof}
We use the formula~\cite{Olver}
\begin{equation}
\GaussF{-m}{b}{c}{z}=\frac{(b)_m}{(c)_m}(-z)^m\GaussF{-m}{1-c-m}{1-b-m}{z^{-1}},
\label{eq:transform}
\end{equation}
where $m$ is a nonnegative integer.

\ref{itm:cor1}
Equation~\eqref{eq:cor-1-1} is the same as Eq.~\eqref{eq:thm1-1} except for the range of $n$.
When $K\le N/2$, the range $n\le N/2$ is included in $n\le N-K$, and Eq.~\eqref{eq:cor-1-1} is proven.

To obtain Eq.~\eqref{eq:cor-1-2}, we apply the formula~\eqref{eq:transform} to Eq.~\eqref{eq:thm1-1} with $m=K$, $b=-n$, and $c=N-K-n+1$:
\begin{align}
&\frac{(N-n)!(N-K)!}{N!(N-K-n)!}\GaussF{-n}{-K}{N-K-n+1}{z}\nonumber\\
&=\frac{(N-n)!)N-K)!}{N!(N-K-n)!}\frac{(-n)_K}{(N-K-n+1)_K}(-z)^K\GaussF{n-N}{-K}{n-K+1}{z^{-1}}.
\label{eq:proof1}
\end{align}
By substituting relations
\begin{align*}
(-n)_K&=(-n)(-n+1)\cdots(-n+K-1)=(-1)^K\frac{n!}{(n-K)!},\\
(N-K-n+1)_K&=(N-K-n+1)(N-K-n+2)\cdots(N-n)=\frac{(N-n)!}{(N-K-n)!},
\end{align*}
we obtain
\[
G_X(z)=\frac{n!(N-K)!}{N!(n-K)!}z^K \GaussF{n-N}{-K}{n-K+1}{z^{-1}}.
\]
The range of $n$ where hypergeometric functions in Eq.~\eqref{eq:proof1} are properly defined is $K\le n\le N-K$.
Therefore, Eq.~\eqref{eq:cor-1-2} for $N/2\le n\le N-K$ is proven.

Equation~\eqref{eq:cor-1-2} for $n\ge N-K$ can be derived from Eq.~\eqref{eq:thm1-2}.
By using Eq.~\eqref{eq:transform} with $m=N-n$, $b=K-N$, and $c=n+K-N+1$, Eq.~\eqref{eq:thm1-2} becomes
\begin{align*}
&\frac{n!K!}{N!(n+K-N)!}\frac{(K-N)_{N-n}}{(n+K-N+1)_{N-n}}(-1)^{N-n}z^K\GaussF{n-N}{-K}{n-K+1}{z^{-1}}\\
&=\frac{n!(N-K)!}{N!(n-K)!}z^K\GaussF{n-N}{-K}{n-K+1}{z^{-1}},
\end{align*}
where
\[
(K-N)_{N-n}=(-1)^{N-n}\frac{(N-K)!}{(n-K)!},\quad
(n+K-N+1)_{N-n}=\frac{K!}{(n+K-N)!}.
\]

\ref{itm:cor2}
Equation~\eqref{eq:cor-2-2} can be directly obtained from Eq.~\eqref{eq:thm1-2}, as Eq.~\eqref{eq:cor-1-1} was derived in \ref{itm:cor1} from Eq.~\eqref{eq:thm1-1}.
As above, Equation~\eqref{eq:cor-2-1} is derived from Eq.~\eqref{eq:thm1} by using the formula~\eqref{eq:transform}, where Eqs.~\eqref{eq:thm1-1} and \eqref{eq:thm1-2} covers $n\le N-K$ and $N-K\le n\le N/2$, respectively.
\end{proof}

\begin{figure}
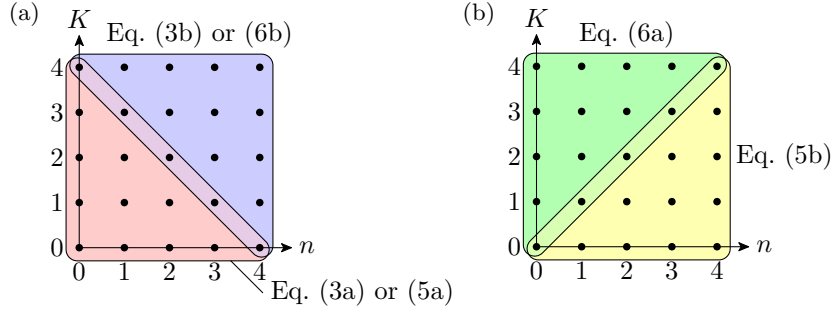
\centering
\raisebox{19mm}{(a)}
\begin{minipage}[c]{54mm}
\includegraphics[clip]{fig1a.mps}
\end{minipage}
\raisebox{19mm}{(b)}
\begin{minipage}[c]{44mm}
\includegraphics[clip]{fig1b.mps}\\
\vspace{3.47mm}
\end{minipage}
\caption{
Covering regions of Eqs.~\eqref{eq:thm1}, \eqref{eq:cor-1}, and \eqref{eq:cor-2} for $N=4$ on the $n$--$K$ plane.
}
\label{fig1}
\end{figure}

\begin{remark}\label{remark:overlap}
In Corollary~\ref{cor}, ``$n\le N/2$'' and ``$n\ge N/2$'' are used to highlight the symmetry of $n$ and $N-n$.
These ranges can be extended as long as the third argument of the hypergeometric function is a positive integer.
For example, Eqs.~\eqref{eq:cor-1-1} and \eqref{eq:cor-1-2} are valid for $n\le N-K$ and $n\ge K$, respectively; both equations include $K\le n\le N-K$, and they become identical.
Similarly, the ranges of $n$ in Eqs.~\eqref{eq:cor-2-1} and \eqref{eq:cor-2-2} can be extended to $n\le K$ and $n\ge N-K$, respectively.
Hence, Eqs.~\eqref{eq:cor-1-1} and \eqref{eq:cor-2-2} are essentially identical to Eqs.~\eqref{eq:thm1-1} and \eqref{eq:thm1-2}, respectively.
Figure~\ref{fig1} shows regions where the equations in Theorem~\ref{thm1} and Corollary~\ref{cor} are defined.
In Fig.~\ref{fig1}(a), the overlap of two regions consists of the points on the descending diagonal $n=N-K$.
This overlap is mentioned in Remark~\ref{remark:1.2}.
In Fig.~\ref{fig1}(b), the overlap of Eqs.~\eqref{eq:cor-1-2} and \eqref{eq:cor-2-1} occurs along the ascending diagonal $n=K$.
Furthermore, when $N$ is even, the $n=K=N/2$ case is included in all four regions, and $G_X(z)$ is the same regardless of which equation is used.
Specifically, when $N=2m$ and $N=K=m$, the probability-generating function $G_X(z)$ can be written using the Legendre polynomial $P_m$:
\[
G_X(z)=\frac{(m!)^2}{(2m)!}\GaussF{-m}{-m}{1}{z}=\frac{(m!)^2}{(2m)!}(z-1)^m\GaussF{-m}{m+1}{1}{(1-z)^{-1}}=\frac{(m!)^2}{(2m)!}(z-1)^m P_m\left(\frac{z+1}{z-1}\right),
\]
where the formula~\cite{Olver}
\[
\GaussF{-m}{b}{c}{z}=\frac{(b)_m}{(c)_m}(1-z)^m\GaussF{-m}{c-b}{1-b-m}{(1-z)^{-1}}
\]
is used in the third equality, and the hypergeometric representation~\cite{Olver} of the Legendre polynomial
\[
P_m(x)=\GaussF{-m}{m+1}{1}{\frac{1-x}{2}}
\]
is used in the last equality.
\end{remark}

\section{Moments}
A typical application of the probability- and moment-generating functions is the calculation of moments.
In many books, e.g., \cite{Johnson, Balakrishnan, Mukhopadhyay, Bishop, Patel}, the mean and variance of the hypergeometric distribution become
\[
\E{X}=\frac{nK}{N},\quad
\Var{X}=\frac{n(N-n)K(N-K)}{N^2(N-1)},
\]
respectively.
More generally, factorial moments, $\E{X(X-1)\cdots(X-r+1)}$ for $r=1, 2,\ldots$, can be exactly calculated for the hypergeometric distribution:
\begin{equation}
\E{X(X-1)\cdots(X-r+1)}=\frac{n!K!(N-r)!}{N!(K-r)!(N-r)!}.
\label{eq:factorial}
\end{equation}
Although direct computation of Eq.~\eqref{eq:factorial} based on the probability mass function~\eqref{eq:pmf} is possible, the factorial moments can be systematically calculated using the probability-generating function $G_X$ as
\[
\E{X(X-1)\cdots(X-r+1)}=\left.\frac{d^r G_X(z)}{dz^r}\right|_{z=1}.
\]

Using the probability-generating function in Eq.~\eqref{eq:thm1-1}, the factorial moment in Eq.~\eqref{eq:factorial} can be easily obtained using formulas (see~\cite{Olver}):
\[
\frac{d^r}{dz^r}\GaussF{a}{b}{c}{z}=\frac{(a)_r(b)_r}{(c)_r}\GaussF{a+r}{b+r}{c+r}{z}
\]
and
\[
\GaussF{a}{b}{c}{1}=\frac{\Gamma(c)\Gamma(c-a-b)}{\Gamma(c-a)\Gamma(c-b)}.
\]

What computation is performed for $n\ge N-K+1$ in which Eq.~\eqref{eq:thm1-2} is applied?
The effective formula~\cite{Olver} in this case is
\begin{equation}
\frac{d^r}{dz^r}(z^{c-1}\GaussF{a}{b}{c}{z})=(c-r)_r z^{c-r-1}\GaussF{a}{b}{c-r}{z}.
\label{eq:diff}
\end{equation}
For $\GaussF{n-N}{K-N}{n+K-N+1}{z}$ ($a=n-N$, $b=K-N$, and $c=n+K-N+1$) in Eq.~\eqref{eq:thm1-2}, $c-r=n+K-N+1-r$ becomes $0$ or a negative integer when $r\ge n+K-N+1$.
Therefore, $\GaussF{a}{b}{c-r}{z}$ cannot be defined for such $r$, and a special treatment is necessary.

\begin{proposition}[Supplement formula to Eq.~\eqref{eq:diff}]\label{prop1}
For positive integers $m$ and $r$ which satisfy $m\le r$,
\[
\frac{d^r}{dz^r}(z^{m-1}\GaussF{a}{b}{m}{z})
=\frac{(m-1)!(a)_{r-m+1}(b)_{r-m+1}}{(r-m+1)!}\GaussF{a+r-m+1}{b+r-m+1}{r-m+2}{z}.
\]
\end{proposition}
\begin{proof}
Using the gamma function instead of the Pochhammer symbol, Eq.~\eqref{eq:diff} can be written as
\[
\frac{d^r}{dz^r}(z^{c-1}\GaussF{a}{b}{c}{z})=\frac{\Gamma(c)}{\Gamma(c-r)}z^{c-r-1}\GaussF{a}{b}{c-r}{z}.
\]
In taking the $c\to m$ limit, Eq.~\eqref{eq:scaled} is used because $m-r$ is a nonpositive integer:
\[
\lim_{c\to m}\frac{\GaussF{a}{b}{c-r}{z}}{\Gamma(c-r)}=\frac{(a)_{r-m+1}(b)_{r-m+1}}{(r-m+1)!}z^{r-m+1}\GaussF{a+r-m+1}{b+r-m+1}{r-m+2}{z}.
\]
\end{proof}

In the computation of the $r$th factorial moment of $X$ for $n\ge N-K+1$, Eq.~\eqref{eq:diff} can be directly applied for $r\le n+K-N$, and Proposition~\ref{prop1} is needed for $r\ge  n+K-N+1$.
In both cases, the final formula of the factorial moment becomes identical with the existing formula~\eqref{eq:factorial}.
Therefore, the moments are not affected when starting from the new formula~\eqref{eq:thm1-2} for $n\ge N-K+1$.
The author presumes that Eq.~\eqref{eq:thm1-2} has not focused thus far because the moments become the same as the formulas obtained from the existing expression~\eqref{eq:thm1-1}.
Moreover, individual values of generating functions have probably been less frequently required, which has kept researchers away from careful investigation of the generating functions of the hypergeometric distribution.

\bibliographystyle{cpclike} 
\bibliography{refs}
\end{document}